\documentclass{amsart}

\usepackage{amsthm,amssymb,amsmath,hyperref}
\usepackage{mathrsfs}
\usepackage{graphicx}
\usepackage{color}
\usepackage{tikz}
\usepackage[utf8]{inputenc}
\usepackage[all]{xy}
\usepackage{mathtools}

\makeatletter
\newcommand{\rmnum}[1]{\romannumeral #1}
\newcommand{\Rmnum}[1]{\expandafter\@slowromancap\romannumeral #1@}
\makeatother

\theoremstyle{plain}

\newtheorem{thm}{Theorem}[section]

\newtheorem{cor}[thm]{Corollary}
\newtheorem{lem}[thm]{Lemma}

\newtheorem{prop}[thm]{Proposition}

\newtheorem{ques}[thm]{Question}
\theoremstyle{definition}
\newtheorem{defn}[thm]{Definition}

\theoremstyle{remark}
\newtheorem{rem}[thm]{Remark}

\theoremstyle{plain}

\numberwithin{equation}{section}

\DeclareMathOperator{\ass}{Ass}
\DeclareMathOperator{\im}{im}
\DeclareMathOperator{\depth}{depth}
\DeclareMathOperator{\grade}{grade}
\DeclareMathOperator{\Hom}{Hom}
\DeclareMathOperator{\height}{ht}
\DeclareMathOperator{\pd}{pd}
\DeclareMathOperator{\rank}{rank}
\DeclareMathOperator{\soc}{soc}
\DeclareMathOperator{\spec}{Spec}

\DeclareMathOperator{\tor}{Tor}

\def\udot#1{\ifmmode\oalign{$#1$\crcr\hidewidth.\hidewidth
    }\else\oalign{#1\crcr\hidewidth.\hidewidth}\fi}

\begin{document}
\title[]{On Generalized Deformation Problems}
\author{Qiurui Li}
\address{Department of Mathematics, Purdue University, 150 N. University St., W. Lafayette, IN 47907, U.S.A.}%
\email{li2889@purdue.edu}

\begin{abstract}
Let $(R,\mathfrak{m})$ be a Noetherian local ring, and $I$ an $R$-ideal such that $\pd_R R/I<\infty$. If $R/I$ satisfies a property $\mathcal{P}$, it is natural to ask if $R$ would also have the property $\mathcal{P}$, we call this the generalized deformation problem. Our paper gives some properties that hold for the generalized deformation problems. There are two main parts in this paper. The first part is about $F$-singularities in the generalized deformation problems. Motivated by Aberbach's work, we show that if every maximal regular sequence on $R/I$ is Frobenius closed, then every regular sequence on $R$ is Frobenius closed. Moreover, under mild assumptions, Frobenius closed can be replaced by tightly closed. By these two results, we solve the generalized deformation problems for $F$-injectivity in the Cohen-Macaulay case and $F$-rationality under mild assumptions. Namely, if $R$ is a Noetherian local ring of characteristic $p$, and $I$ an $R$-ideal such that $\pd_R R/I<\infty$, then if $R$ is Cohen-Macaulay and $R/I$ is $F$-injective, then $R$ is $F$-injective. If $R$ is excellent and $R/I$ is $F$-rational, then $R$ is $F$-rational. The second part is about some basic properties of rings in arbitrary characteristic. We prove that if $R$ is a Noetherian local ring, $I$ an $R$-ideal such that $\pd_R R/I<\infty$, then if $R/I$ is $S_k$, $R_k+S_{k+1}$, normal, reduced or a domain, then so is $R$.
\end{abstract}

\maketitle

\section{Introduction}
Let us begin by recalling the deformation problem.
\begin{ques}\label{defquestion}
Let $(R,\mathfrak{m})$ be a Noetherian local ring, and $x\in\mathfrak{m}$ be a regular element. If $R/xR$ satisfies the property $\mathcal{P}$, then does $R$ itself satisfy the property $\mathcal{P}$?
\end{ques}
This problem has been studied extensively in characteristic $p>0$ when $\mathcal{P}=$ $F$-singularities, we list the best known progress of the related work.
\begin{enumerate}
\item[(1)] Strongly $F$-regularity fails the deformation property in general \cite{Sin99c}, but it deforms for normal $\mathbb{Q}$-Gorenstein rings \cite{AKM98}.
\item[(2)] $F$-purity fails the deformation property in general \cite{Fed83, Sin99b}, but it deforms for normal $\mathbb{Q}$-Gorenstein rings \cite{HW02, Sch09, PS20}.
\item[(3)] $F$-rationality always deforms \cite{HH94a}.
\item[(4)] Deformation of $F$-injectivity remains an open problem in general. But it is known that $F$-injectivity deforms for Cohen-Macaulay rings \cite{Fed83}, and that $F$-purity always deforms to $F$-injectivity \cite{HMS14}.
\end{enumerate}

More precisely, Singh in \cite{Sin99c} constructed an example where $R$ is a $3$-dimensional Cohen-Macaulay ring of characteristic $p>2$ and there is a regular element $t$ on $R$ such that $R/tR$ is strongly $F$-regular, but under some circumstance, $R$ is not $F$-pure. Since strongly $F$-regularity implies $F$-purity, this example shows the failure of the deformation problems of strongly $F$-regular and $F$-pure.

\smallskip
We also list some progress in some properties that do not have special requirements on the characteristics of the ring:
\begin{enumerate}
\item[(1)] Regularity always deforms \cite[Chapter 0, Corollary 17.1.8]{GD1}.
\item[(2)] Cohen-Macaulay rings always deform \cite[Theorem 2.1.3($a$)]{BH}.
\item[(3)] Gorenstein rings always deform \cite[Proposition 3.1.19($b$)]{BH}.
\item[(4)] Complete intersections always deform \cite[Theorem 2.3.4($a$)]{BH}.
\item[(5)] Normal rings always deform \cite[Proposition \Rmnum{1}.7.4]{Sey}.
\item[(6)] Reduced rings always deform \cite[Proposition 3.4.6]{GD2}.
\item[(7)] $S_k$ always deforms \cite[Lemma 0(\rmnum{1})]{BR}.
\item[(8)] $R_k$ always deforms under the condition that $R$ is equidimensional and catenary \cite[Lemma 2.1]{I}.
\item[(9)] $R_k+S_{k+1}$ always deform \cite[Lemma 0(\rmnum{2})]{BR}.
\item[(10)] Domains always deform \cite[Proposition 3.4.5]{GD2}.
\end{enumerate}

\smallskip
Clearly if $x$ is regular on $R$, we have that $\pd_R R/xR=1<\infty$, so this leads to a natural generalization of the deformation question:

\begin{ques}[Generalized Deformation Problem]
Let $(R,\mathfrak{m})$ be a Noetherian local ring, $I\subseteq R$ an ideal such that $\pd_R R/I<\infty$. If $R/I$ satisfies the property $\mathcal{P}$, then does $R$ itself satisfy $\mathcal{P}$?
\end{ques}

In the first part of this paper (Section 3 and Section 4) we still focus on $\mathcal{P}$ defined by the Frobenius. Since the deformation question fails for strongly $F$-regularity and $F$-purity by Singh's counterexample, we focus our attention to the cases when $\mathcal{P}=$ $F$-injectivity and $\mathcal{P}=$ $F$-rationality in this paper. In the second part of this paper (Section 5) we focus on the last six cases from the ten items we listed above since
\begin{enumerate}
\item[(1)] Regularity satisfies the generalized deformation problem \cite[Theorem 6.1(1)]{AFH}.
\item[(2)] Cohen-Macaulay rings satisfy the generalized deformation problem \cite[Theorem 2.1]{AFH}.
\item[(3)] Gorenstein rings satisfy the generalized deformation problem \cite[Theorem 4.3(1)(a)]{AFH}.
\item[(4)] Complete intersections satisfy the generalized deformation problem \cite[Theorem 5.3(\rmnum{1})$\Rightarrow$(\rmnum{3})]{AFH}.
\end{enumerate}

Our main theorems in this paper are the followings:

\begin{thm}[Theorem \ref{ThmF-in} and Theorem \ref{ThmF-r}]
Let $(R,\mathfrak{m})$ be a Noetherian local ring of characteristic $p>0$, $I\subseteq R$ an ideal such that $\pd_R R/I<\infty$. Then
\begin{enumerate}
\item[(1)] if $R$ is Cohen-Macaulay, and $R/I$ is $F$-injective, then $R$ is $F$-injective.
\item[(2)] if $R$ is excellent, and $R/I$ is $F$-rational, then $R$ is $F$-rational.
\end{enumerate}
\end{thm}

\begin{thm}[Theorem \ref{r}, Theorem \ref{s}, Theorem \ref{sandr}, Corollary \ref{reduced}, Corollary \ref{normal} and Theorem \ref{domain}]
Let $(R,\mathfrak{m})$ be a Noetherian local ring, $I\subseteq R$ an ideal such that $\pd_R R/I<\infty$. Then
\begin{enumerate}
\item[(1)] if $R$ is equidimensional and catenary, $I$ is a perfect ideal and $R/I$ is $R_k$, then $R$ is $R_k$.
\item[(2)] if $R/I$ is $S_k$, then $R$ is $S_k$.
\item[(3)] if $R/I$ is $R_k$ and $S_{k+1}$, then $R$ is $R_k$ and $S_{k+1}$.
\item[(4)] if $R/I$ is reduced, then $R$ is reduced.
\item[(5)] if $R/I$ is normal, then $R$ is normal.
\item[(6)] if $R/I$ is a domain, then $R$ is a domain.
\end{enumerate}
\end{thm}

This paper is organized as followings:

\smallskip
In section 2, we introduce the basic definitions and notations that are needed in our paper. In section 3 we aim to prove two important propositions \ref{prop1} and \ref{prop2} that will be used in the proofs of our two theorems in section 4. To achieve this, we apply the methods in the proof of \cite[Lemma 1.2]{A} to the case of Frobenius closure, then use the method of proving (6)$\Rightarrow$(2) in \cite[Theorem 1.1]{A} to get Proposition \ref{prop1}. We use a similar method to get Proposition \ref{prop2}. After all preparations, in section 4 we complete the proofs of our main theorems for $F$-singularities. In section 5 we firstly prove two lemmas (i.e. Lemma \ref{new1} and Lemma \ref{new2}) to help to solve the case of $I\nsubseteq P$ for $R_k$ and $S_k$ where $P$ is a prime ideal of $R$, then we prove three theorems about Serre's conditions $R_k$,  $S_k$ and $R_k+S_{k+1}$ in the deformation problem, which help to solve the cases for normal rings and reduced rings, and at last we use Auslander's Zerodivisor Conjecture (which is a theorem by Robert's work \cite{R2}) to prove the theorem about domains in the generalized deformation problem.
\bigskip

\section{Preliminaries}
\begin{defn}
Let $p$ be a prime number and $R$ a commutative ring such that $\mathbb{F}_p=\mathbb{Z}/p\mathbb{Z}\subseteq R$. We say that $R$ has characteristic $p$. Let $F:R\rightarrow R$ be defined as $F(r)=r^p$ for any $r\in R$, then it is a ring endomorphism, called the Frobenius endomorphism. Let $\mathcal{F}$ be Peskine-Szpiro's Frobenius base change functor from $R$-modules to $R$-modules, and denote $\mathcal{F}^e$ as the $e$-fold iterated composition of $\mathcal {F}$.
\end{defn}

\begin{rem}
\begin{enumerate}
\item[(1)] $\mathcal{F}^e(R)=R$ for any $e>0$.

\item[(2)] $\mathcal{F}^e$ commutes with arbitrary direct sum and arbitrary direct limits.

\item[(3)] $\mathcal{F}^e(R/I)=R/I^{[q]}$, where $q=p^e$ and $I^{[q]}=\{x^q\arrowvert \,x\in I\}$.

\item[(4)] For every $R$-module $M$ the natural map $\varphi:M\rightarrow \mathcal{F}^e(M)$ is such that for all $r\in R$ and all $u\in M$, $\varphi(ru)=r^q\varphi(u)$. Because of this, we shall use $u^q$ to denote the image of $u$ under this map.
\end{enumerate}
\end{rem}

\begin{defn}\label{defn2.3}
Let $N\subseteq M$ be two $R$-modules, notice that the map $\mathcal{F}^e(N)\rightarrow \mathcal{F}^e(M)$ is the base change $F^e:R\rightarrow R$ of the inclusion map $N\hookrightarrow M$. The map $\mathcal{F}^e(N)\rightarrow \mathcal{F}^e(M)$ may not be injective, and we denote the image of $\mathcal{F}^e(N)$ of the induced map as $N_M^{[q]}$ or $N^{[q]}$ if $M$ is clear from context. Let $\psi:N\rightarrow M$ be an $R$-linear map, then the induced map $\psi^{[q]}:\mathcal{F}^e(N)\rightarrow \mathcal{F}^e(M)$ is defined as
$$\psi^{[q]}(x)=(\psi(x))^q.$$
\end{defn}

\begin{defn}
Let $R$ be a Noetherian ring of characteristic $p>0$, and  $R^{\circ}$ be the complement of the union of all minimal primes of $R$. If $I$ is an $R$-ideal, we define the tight closure $I^*$ of $I$ to consist all elements $x\in R$ such that there exists $c\in R^{\circ}$ such that $cx^q\in I^{[q]}$ for all $q\gg 0$. Obviously $I\subseteq I^*$ since we can take $c=1$, and an ideal $I$ is called tightly closed if $I^*=I$. For $R$-modules $N\subseteq M$, the tight closure $N_M^*$ of $N$ in $M$ consists of elements $u\in M$ such that for some $c\in R^{\circ}$,
$$cu^q\in N_M^{[q]}\subseteq \mathcal{F}^e(M)$$
for all $q\gg 0$. Similarly the Frobenius closure $I^F$ of $I$ is the set of all elements $x\in R$ such that there exists $q=p^e>0$ such that $x^q\in I^{[q]}$. An ideal $I$ is called Frobenius closed if $I^F=I$. Sometimes we may say that a sequence is tightly closed or Frobenius closed to mean the ideal generated by the sequence is tightly closed or Frobenius closed for simplicity.
\end{defn}

\begin{defn}
For an $n\times m$ matrix $\varphi$ with entries in $R$, we write $I_t(\varphi)$ for the $R$-ideal generated by all $t\times t$ minors of $\varphi$. We use the convention that $I_t(\varphi)=R$ for $t\leqslant 0$ and $I_t(\varphi)=0$ for $t>\min\{m,n\}$. We define the rank of $\varphi$ to be $\rank(\varphi)=\max\{r\arrowvert \,I_r(\varphi)\neq 0\}$.
\end{defn}

\begin{defn} Let $R$ be a ring and $M$ an $R$-module, then we define the dual of $M$ to be $\Hom_R(M,R)$, denoted as $M^{\vee}$. If we have $\varphi: R^m\rightarrow R^n$ an $R$-linear map, then there is an induced map $\varphi^{\vee}: (R^n)^{\vee}\rightarrow (R^m)^{\vee}$ where $\varphi^{\vee}$ is the transpose of $\varphi$ as matrices. Note that $(R^l)^{\vee}$ is isomorphic to $R^l$ for any $l>0$.
\end{defn}

\begin{defn}\label{def2.7}
We say that a Noetherian local ring $(R,\mathfrak{m})$ is $F$-rational if it is a homomorphic image of a Cohen-Macaulay ring and every ideal generated by a system of parameters is tightly closed.
\end{defn}

\begin{rem}\label{rem2.8}
Actually an $F$-rational local ring is Cohen-Macaulay (see, e.g., \cite[Theorem on Page 125]{H}).
\end{rem}

Let us recall the Frobenius action on the top local cohomology. Let $(R,\mathfrak{m})$ be a Noetherian local ring of characteristic $p$, and $\dim R=d$. Let $x_1,\ldots,x_d$ be a system of parameters of $R$, then the top local cohomology module $H_{\mathfrak{m}}^d(R)$ is isomorphic to
\begin{equation*}
\frac{R_{x_1\ldots x_d}}{\sum_i \im\,(R_{x_1\ldots \widehat{x_i}\ldots x_d})},
\end{equation*}
where $\sum_i \im\,(R_{x_1\ldots \widehat{x_i}\ldots x_d})$ is the image of the last map
\begin{equation*}
R_{x_2\ldots x_d}\oplus\ldots\oplus R_{x_1\ldots x_{d-1}}\longrightarrow R_{x_1\ldots x_d}
\end{equation*}
in the \v{C}ech complex. An element $y\in H_{\mathfrak{m}}^d(R)$ is the homology class of $\frac{a}{x^t}$, denoted as $[\frac{a}{x^t}]$, where $a\in R$ and $x=x_1\ldots x_d$. Notice that $F([\frac{a}{x^t}])=[\frac{a^p}{x^{tp}}]$, and if $R$ is Cohen-Macaulay, then $[\frac{a}{x^t}]=0$ if and only if $a\in (x_1^t,\ldots,x_d^t)$. The Frobenius actions on other local cohomologies can be described explicitly using \v{C}ech complex in a similar way.

\begin{rem}\label{rem2.9}
Smith gave an important characterization of $F$-rationality in terms of local cohomology, i.e. a $d$-dimensional Noetherian local ring $(R,\mathfrak{m})$ is $F$-rational if and only if $R$ is Cohen-Macaulay and $0_{H_{\mathfrak{m}}^d(R)}^*=0$.
\end{rem}

\begin{defn}
A Noetherian local ring $(R,\mathfrak{m})$ of charateristic $p>0$ is called $F$-injective if the natural Frobenius action on $H_{\mathfrak{m}}^i(R)$ is injective for all $i$.
\end{defn}

\begin{rem}\label{rem2.11}
Note that $F$-rational implies $F$-injective by \cite[Proposition A.3.(ii)]{DM}.
\end{rem}

\begin{defn}
A Noetherian local ring $R$ is said to be equidimensional if for $P$ to be any minimal prime ideal of $R$, we have $\dim\,R/P=\dim\,R$.
\end{defn}

\begin{defn}
A Noetherian local ring $R$ is said to be catenary if for any two prime ideals $P\subsetneq Q$ in $R$, there is a chain of prime ideals $P=P_0\subsetneq P_1\subsetneq\ldots\subsetneq P_n=Q$ that cannot be refined any further, and any such chain has the same length.
\end{defn}

\begin{prop}\label{e+c}\cite[Lemma 2 on Page 250]{M}
Let $(R,\mathfrak{m})$ be a Noetherian local ring, and $R$ is equidimensional and catenary, then for any $P\in \spec(R)$, $\dim\,R=\height\,P+\dim\,R/P$.
\end{prop}

\begin{defn}
Let $R$ be a Noetherian ring and $I$ an $R$-ideal, then $I$ is said to be perfect if $\grade\,I=\pd_R R/I$.
\end{defn}

\begin{defn}
Let $R$ be a Noetherian ring, $k\geqslant 0$ is an integer, then $R$ satisfies Serre's condition $R_k$ if $R_P$ is regular whenever $P\in\spec(R)$ and $\dim\,R_P\leqslant k$. And $R$ satisfies Serre's condition $S_k$ if $\depth\,R_P\geqslant\min\{k,\dim\,R_P\}$ for any $P\in\spec(R)$.
\end{defn}
\bigskip

\section{Proof of the Main Propositions}
In this section we prove our main propositions \ref{prop1} and \ref{prop2}. We first prove Theorem \ref{LG(1)(2)} which can be regarded as an analog of the following result of Aberbach for Frobenius closure.

\begin{thm}\cite[Lemma 1.2]{A}\label{Alem1.2}
Let $R$ be a Noetherian ring of characteristic $p>0$. Let $(L_{\bullet},a_{\bullet})$ and $(G_{\bullet},b_{\bullet})$ be complexes of finitely generated free modules of length $d$ where $L_{\bullet}$ is acyclic. If $\phi_{\bullet}:G_{\bullet}\rightarrow L_{\bullet}$ is a chain map and $\im\,\phi_0\subseteq (\im\,a_1)_{L_0}^*$, then
\begin{enumerate}
\item[(1)] $\im\,\phi_d\subseteq (I_1(b_d)L_d)_{L_d}^*$, and
\item[(2)] $\im\,\phi_d^{\vee}\subseteq (\im\,b_d^{\vee})_{G_d^{\vee}}^*$.
\end{enumerate}
\end{thm}

To prove its Frobenius closure version Theorem \ref{LG(1)(2)}, we also need the following:

\begin{thm}[Buchsbaum-Eisenbud]\label{BEacyclic}
Let $(R,\mathfrak{m})$ be a Noetherian local ring,
\begin{equation*}
G_{\bullet}:0\rightarrow R^{\oplus b_n}\stackrel{\alpha_n}{\longrightarrow} R^{\oplus b_{n-1}}\rightarrow\ldots\rightarrow R^{\oplus b_1}\stackrel{\alpha_1}{\longrightarrow} R^{\oplus b_0}\rightarrow 0
\end{equation*}
is a finite free complex over $R$. Let $r_i=\rank(\alpha_i)=\max\{r\arrowvert \,I_r(\alpha_i)\neq 0\}$. Then $G_{\bullet}$ is acyclic if and only if the following two conditions hold:
\begin{enumerate}
\item[(1)] $r_i+r_{i+1}=b_i$ for any $i$ such that $1\leqslant i\leqslant n$.
\item[(2)] $\depth_{I_{r_i}(\alpha_i)}R\geqslant i$ for any $i$ such that $1\leqslant i\leqslant n$.
\end{enumerate}
\end{thm}
This theorem can be found on \cite[Corollary 1]{BE}, and we will use it in the proof of Theorem \ref{LG(1)(2)} to show that the Frobenius action on an exact sequence would not change the exactness of the sequence.

\begin{thm}\label{LG(1)(2)}
Let $R$ be a Noetherian ring of characteristic $p>0$. Let $(L_{\bullet},a_{\bullet})$ and $(G_{\bullet},b_{\bullet})$ be complexes of finitely generated free modules of length $d$ where $L_{\bullet}$ is acyclic. If $\phi_{\bullet}:G_{\bullet}\rightarrow L_{\bullet}$ is a chain map and $\im\, \phi_0\subseteq (\im\, a_1)_{L_0}^F$, then
\begin{enumerate}
\item[(1)]$\im\, \phi_d\subseteq (I_1(b_d)L_d)^F$, and
\item[(2)]$\im\, \phi_d^{\vee}\subseteq (\im\, b_d^{\vee})^F$.
\end{enumerate}
\end{thm}
\begin{proof}
Our proof is similar to the argument in \cite[Lemma 1.2]{A}.
\smallskip
\newline \indent For (1), since $\im\, \phi_0\subseteq (\im\, a_1)_{L_0}^F$, there exists $q=p^e>0$, such that $(\im\, \phi_0)^{[q]}\subseteq (\im\, a_1)^{[q]}$, i.e. $\im\, (\phi_0^{[q]})\subseteq \im\, (a_1^{[q]})$. So we have the following commutative diagram
\begin{displaymath}
\xymatrix{
0\ar[r]& \mathcal{F}^e(L_d)\ar[r]^{a_d^{[q]}}& \mathcal{F}^e(L_{d-1})\ar[r]^{a_{d-1}^{[q]}}& \ldots\ar[r]^{a_1^{[q]}}& \mathcal{F}^e(L_0)\ar[r]& H_0(\mathcal{F}^e(L_{\bullet}))\ar[r]& 0 \\
0\ar[r]& \mathcal{F}^e(G_d) \ar[u]^{\phi_d^{[q]}} \ar[r]^{b_d^{[q]}}& \mathcal{F}^e(G_{d-1})\ar[u]^{\phi_{d-1}^{[q]}} \ar[r]^{b_{d-1}^{[q]}}& \ldots\ar[r]^{b_1^{[q]}}& \mathcal{F}^e(G_0)\ar[u]^{\phi_0^{[q]}} \ar[r]& H_0(\mathcal{F}^e(G_{\bullet}))\ar[u]^0 \ar[r]&0}
\end{displaymath}
and we claim that the first row is exact.
\smallskip
\newline \indent By our condition,
$$0\rightarrow L_d\stackrel{a_d}{\longrightarrow} L_{d-1}\rightarrow\ldots\rightarrow L_1\stackrel{a_1}{\longrightarrow} L_0\rightarrow 0$$
is acyclic, so by Theorem \ref{BEacyclic}, we have $\rank(a_i)+\rank(a_{i+1})=\rank(L_i)$ and $\depth_{I_{r_i}(a_i)}R\geqslant i$ for any $1\leqslant i\leqslant d$. By Definition \ref{defn2.3}, the entries of $a_i^{[q]}$ are just the $q$-th power of entries of $a_i$, and if $x_1,\ldots, x_c$ is regular in a local ring, then $x_1^t,\ldots, x_c^t$ is also regular for any $t>0$, we have that $\depth_{I_{r_i}(a_i^{[q]})} R =\depth_{I_{r_i}(a_i)} R\neq 0$. And every $r_i\times r_i$ minor of $a_i^{[q]}$ is the $q$-th power of the $r_i\times r_i$ minor of $a_i$. But since $\depth_{I_{r_i}(a_i)}R\geqslant i>0$, there is at least one $r_i\times r_i$ minor of $a_i$ that is a regular element on $R$, so its $q$-th power is still a regular element, so $r_i\leqslant \rank(a_i^{[q]})$. But it is also obvious that $\rank(a_i^{[q]})\leqslant \rank(a_i)$, we have that $r_i=\rank(a_i^{[q]})$. So $$\rank(a_i^{[q]})+\rank(a_{i+1}^{[q]})=\rank(L_i)=\rank(\mathcal{F}^e(L_i)).$$
By Theorem \ref{BEacyclic} again, we know that
$$0\rightarrow \mathcal{F}^e(L_d)\stackrel{a_d^{[q]}}{\longrightarrow} \mathcal{F}^e(L_{d-1})\rightarrow\ldots\rightarrow \mathcal{F}^e(L_1)\stackrel{a_1^{[q]}}{\longrightarrow} \mathcal{F}^e(L_0)\rightarrow 0$$
is acyclic, so the claim is true.
\smallskip
\newline \indent Since $\im\,\phi_0^{[q]}\subseteq \im\,a_1^{[q]}$, the chain maps are homotopic to the zero map. Hence we get a commutative diagram
\begin{displaymath}
\xymatrix{
\mathcal{F}^e(L_d) \\
\mathcal{F}^e(G_d)\ar[u]^{\phi_d^{[q]}} \ar[r]^{b_d^{[q]}}& \mathcal{F}^e(G_{d-1})\ar[ul]_{\exists\ \psi}
}
\end{displaymath}
where $\psi$ is an $R$-linear map from $\mathcal{F}^e(G_{d-1})$ to $\mathcal{F}^e(L_d)$. Thus $\phi_d^{[q]}=\psi\circ b_d^{[q]}$, thus $\im\phi_d^{[q]}\subseteq I_1(b_d^{[q]})\mathcal{F}^e(L_d)=(I_1(b_d)L_d)^{[q]}$. Hence $\im\phi_d\subseteq (I_1(b_d)L_d)^F$.
\smallskip
\newline \indent For (2), dualizing the above diagram, we get
\begin{displaymath}
\xymatrix{
\mathcal{F}^e(L_d)^{\vee}\ar[d]_{\phi_d^{[q]\vee}} \ar[rd]^{\psi^{\vee}}\\
\mathcal{F}^e(G_d)^{\vee}& \mathcal{F}^e(G_{d-1})^{\vee} \ar[l]_{b_d^{[q]\vee}}.
}
\end{displaymath}
Thus $\phi_d^{[q]\vee}=b_d^{[q]\vee}\circ\psi^{\vee}$, thus
$$\im(\,\phi_d^{\vee})^{[q]}=\im\,(\phi_d^{[q]\vee})\subseteq \im\,(b_d^{[q]\vee})=(\im\,b_d^{\vee})^{[q]}.$$
Hence $\im\,\phi_d^{\vee}\subseteq (\im\,b_d^{\vee})^F$.

\end{proof}

\begin{lem}\cite[Lemma 1.3]{A}\label{lemma3.3}
Let $R$ be a commutative ring. Let $(L_{\bullet},d_{\bullet})$ and $(G_{\bullet},e_{\bullet})$ be length $d$ complexes of finitely generated free modules such that $G_{\bullet}^{\vee}$ is acyclic. Let $\phi_{\bullet}:G_{\bullet}\rightarrow L_{\bullet}$ be a chain map. If $\im\,\phi_d^{\vee}\subseteq \im\,e_d^{\vee}$, then $\im\,\phi_0\subseteq \im\,d_1$.
\end{lem}

\begin{lem}\label{lemDF}
Let $(R,\mathfrak{m})$ be a Noetherian local ring of characteristic $p>0$. $I$ an $R$-ideal. Let $S=R/I$ and $J$ an $R$-ideal that contains $I$. If for $\bar{J}$ the image of $J$ in $S$, we have that $\bar{J}$ is a Frobenius closed $S$-ideal, then $J$ is a Frobenius closed $R$-ideal, i.e. any $R$-ideal that contains $I$ is Frobenius closed.
\end{lem}
\begin{proof}
Let $R\rightarrow R/I$ be the natural projection map. Let $x\in J_R^F$ where $J_R^F$ is denoted as the Frobenius closure of $J$ as an $R$-ideal, then there exists $q=p^e>0$ such that $x^q\in J^{[q]}$. Let $\bar{x}$ be the image of $x$ in $R/I$. Then $\Bar{x}^q\in \bar{J}^{[q]}$. Since $\bar{J}$ is Frobenius closed as an $S$-ideal, we get that $\bar{x}\in\bar{J}$, i.e. $x\in J$. Hence $J_R^F=J$.
\end{proof}

Next, we need to prove that if $\pd_R R/I<\infty$, and $\underline{x}=x_1,\ldots,x_c\in \mathfrak{m}-I$ is a maximal regular sequence on $R/I$, then $\pd_R R/(I+(\underline{x}))<\infty$, which is the precondition to use Auslander-Buchsbaum's formula to the $R$-module $R/(I+(\underline{x}))$. And it suffices to show the case for $c=1$ which is Lemma \ref{lemma3.6}, because we can then use induction to get the general case.
\smallskip
\newline\indent
We recall the following lemma which is a standard proof.

\begin{lem}\label{lemma3.5}
Let $(R,\mathfrak{m},K)$ be a Noetherian local ring. Given a short exact sequence of $R$-modules
$$0\rightarrow M'\rightarrow M\rightarrow M''\rightarrow 0,$$
if any two of $\{M',M,M''\}$ have finite projective dimensions, then so does the third one.
\end{lem}

\begin{lem} \label{lemma3.6}
Let $(R,\mathfrak{m})$ be a Noetherian local ring and $I$ an $R$-ideal. Suppose $\pd_R R/I<\infty$ and $y\in \mathfrak{m}-I$ is a regular element on $R/I$, then $\pd_R R/(I+(y))<\infty$.
\end{lem}
\begin{proof}
Since $y$ is regular on $R/I$, we have the short exact sequence
$$0\rightarrow R/I\stackrel{\centerdot y}{\rightarrow} R/I\rightarrow R/(I+(y)) \rightarrow 0.$$
And by Lemma \ref{lemma3.5}, we get that $R/(I+(y))$ has finite projective dimension since $\pd_R R/I<\infty$.
\end{proof}

\indent Now we are ready to prove our first proposition.

\begin{prop}\label{prop1}
Let $(R,\mathfrak{m})$ be a Noetherian local Cohen-Macaulay ring of characteristic $p$ and $\dim\,R=d$. Suppose that there exists an $R$-ideal $I$ such that $\pd_R R/I<\infty$ and every maximal regular sequence on $R/I$ is Frobenius closed, then every regular sequence on $R$ is Frobenius closed.
\end{prop}
\begin{proof}
Choose $\underline{x}=x_1,\ldots,x_c$ a maximal regular sequence on $R/I$, then we have $\depth_R R/(I+(\underline{x}))=0$. By Lemma \ref{lemma3.6}, we know that $\pd _R R/(I+(\underline{x}))<\infty$, then by Auslander-Buchsbaum's formula we have that $\pd_R R/(I+(\underline{x}))=\depth\ R$, which we denote as $g$. Let $\underline{y}=y_1,\ldots,y_g$ be a maximal regular sequence on $R$. Let
\begin{equation*}
0\rightarrow G_g\stackrel{a_g}{\longrightarrow} G_{g-1}\rightarrow\ldots\rightarrow G_1\stackrel{a_1}{\longrightarrow} G_0\rightarrow 0
\end{equation*}
be a projective resolution of $R/(I+(\underline{x}))$. Let $K_{\bullet}(\underline{y})$ denote the Koszul complex of $\underline{y}=y_1,\ldots,y_g$, which is a free resolution of $R/(\underline{y})$. We will first show that $(\underline{y})$ is Frobenius closed.
\smallskip
\newline \indent If not, firstly we claim that there exists $w\in (\underline{y})^F-(\underline{y})$ such that $\mathfrak{m}w\subseteq (\underline{y})$. Consider the finitely generated $R/(\underline{y})$-module $M=(\underline{y})^F/(\underline{y})$. Since $\underline{y}$ is maximal regular on $R$, we have $\depth_{R/(\underline{y})}M=0$, which means $\soc_{R/(\underline{y})}M\neq 0$. Hence for any $0\neq \bar{w}\in \soc_{R/(\underline{y})}M$, its lifting $w$ in $R$ satisfies our condition. We then construct a chain map from $G_{\bullet}^{\vee}$ to $K_{\bullet}(\underline{y})$.
\begin{displaymath}
\xymatrix{
0& H_0(G_{\bullet}^{\vee}) \ar[l] \ar[d]_{\exists\ \varphi}& G_{g}^{\vee} \ar[l] \ar[d]_{\phi_0}& G_{g-1}^{\vee} \ar[l]_{a_g^{\vee}} \ar[d]_{\phi_1}& \ldots\ar[l]& G_1^{\vee}\ar[l] \ar[d]_{\phi_{g-1}}& G_0^{\vee}\ar[l]_{a_1^{\vee}} \ar[d]_{\phi_g}& 0\ar[l] \\
0& R/(\underline{y})\ar[l]& R\ar[l]& R^{\oplus g}\ar[l]& \ldots\ar[l]& R^{\oplus g}\ar[l]& R\ar[l]& 0\ar[l]}
\end{displaymath}
\indent Fix a basis for the free module $G_g$, so that we have a corresponding dual basis for $G_g^{\vee}$. Let $e_1$ be an element of the dual basis of $G_g^{\vee}$. Let $\phi_0:G_g^{\vee}\rightarrow R$ be given by extending $\phi_0(e_1)=w$ and $\phi_0(e_i)=0$ for $i>1$. We want to show there exists $\varphi:H_0(G_{\bullet}^{\vee})\rightarrow R/(\underline{y})$ such that the left block of the diagram commutes, i.e. show that $\phi_0(\im\,a_g^{\vee})\subseteq (\underline{y})$. Since $(G_{\bullet},a_{\bullet})$ is minimal free, all entries of the matrix $a_g$ are in $\mathfrak{m}$, and $a_g^{\vee}$ is just the transpose of $a_g$ as a matrix, so all entries of $a_g^{\vee}$ are also in $\mathfrak{m}$. So it suffices to show $\phi_0(\mathfrak{m}G_g^{\vee})\subseteq (\underline{y})$.
\smallskip
\newline\indent By our definition of $\phi_0$, $\phi_0(G_g^{\vee})=Rw$, so $\phi_0(\mathfrak{m}G_g^{\vee})\subseteq \mathfrak{m}w\subseteq(\underline{y})$. So there is an induced map $\varphi:H_0(G_{\bullet}^{\vee})\rightarrow R/(\underline{y})$ such that the left block commutes. Since each $G_i$ is free, $G_i^{\vee}\cong G_i$ is also free. Since the second row is acyclic, we can get liftings $\phi_i: G_{g-i}^{\vee}\rightarrow R^{\oplus\tbinom{g}{i}}$ of $\phi_0$ such that the diagram commutes. By Theorem \ref{LG(1)(2)}, since $w\in (\underline{y})^F$, i.e. $\im\,\phi_0\subseteq (\im\,(\underline{y}))^F$, we have $\im\,\phi_g^{\vee}\subseteq (\im\,(a_1^{\vee\vee}))^F=(\im\,a_1)^F$. From the condition in our Proposition \ref{prop1} we know that $(\underline{x})$ is Frobenius closed in $R/I$, hence by Lemma \ref{lemDF} we know that $I+(\underline{x})$ is Frobenius closed in $R$. From our construction of the projective resolution of $R/(I+(\underline{x}))$ we know that $\im\,a_1=I+(\underline{x})$, hence $(\im\,a_1)^F=\im\,a_1$, hence $\im\,\phi_g^{\vee}\subseteq \im\,a_1$. Since $G_{\bullet}^{\vee\vee}$ is acyclic because $G_{\bullet}^{\vee\vee}=G_{\bullet}$ which is a projective resolution of $R/(I+(\underline{x}))$, by Lemma \ref{lemma3.3} we get that $\im\,\phi_0\subseteq \im\, (\underline{y})$, i.e. $w\in (\underline{y})$, which is a contradiction to our claim. So $(\underline{y})^F=(\underline{y})$. Now we have proved that any maximal regular sequence in $R$ is Frobenius closed.
\smallskip
\newline\indent Now for $y_1,\ldots,y_{c'}$ being part of a maximal regular sequence where $c'<g$, we can extend $y_1,\ldots,y_{c'}$ to a maximal regular sequence $y_1,\ldots,y_{c'},\ldots,y_g$, then we have
\begin{equation*}
\begin{split}
(y_1,\ldots,y_{c'})^F &\subseteq \bigcap_{n>0} (y_1,\ldots,y_{c'},y_{c'+1}^n,\ldots,y_g^n)^F\\ &= \bigcap_{n>0}(y_1,\ldots,y_{c'},y_{c'+1}^n,\ldots,y_g^n) \\ &= (y_1,\ldots,y_{c'}),
\end{split}
\end{equation*}
\newline where the first equality holds because $y_1,\ldots,y_{c'},y_{c'+1}^n,\ldots,y_g^n$ is also a maximal regular sequence and the second equality holds because of Krull's Intersection Theorem. Hence any ideal generated by a regular sequence is Frobenius closed.
\end{proof}

Now we complete the proof of one of our main propositions. We next prove our second proposition.

\begin{prop}\label{prop2}
Let $(R,\mathfrak{m})$ be a Noetherian excellent local Cohen-Macaulay ring of characteristic $p>0$, $\dim\,R=d$. Suppose there exists $I$ an $R$-ideal such that $\pd_R R/I<\infty$ and every maximal regular sequence on $R/I$ is tightly closed. Then every regular sequence on $R$ is tightly closed.
\end{prop}

We want to change the conditions of Frobenius closure to tight closure in Proposition \ref{prop1}. As we can see, Theorem \ref{LG(1)(2)} has its corresponding version Theorem \ref{Alem1.2} of tight closure. But it is not trivial to change the Frobenius closure to tight closure in Lemma \ref{lemDF}. To make it work, we need to assume some mild conditions, i.e. $R$ is an excellent local ring. We now prove a tight closure analog of Lemma \ref{lemDF}.

\smallskip

The following two lemmas are from Hochster's notes on tight closure.

\begin{lem}\cite[Proposition on Page 51]{H}\label{lemH1}
Let $R$ be a Noetherian ring of characteristic $p>0$, let $N\subseteq M$ be $R$-modules and let $u\in M$. Then $u\in N_M^*$ if and only if the image $\bar{u}$ of $u$ in $M/N$ is in $0_{M/N}^*$.
\smallskip
\newline\indent Hence if we map a free module $G$ onto $M$, say $\varphi:G\rightarrow M$, let $H=\varphi^{-1}(N)\subseteq G$, and let $v\in G$ be such that $\varphi(v)=u$, then $u\in N_M^*$ if and only if $v\in H_G^*$.
\end{lem}

\begin{lem}\cite[Persistence of Tight Closure on Page 183]{H}\label{lemH2}
Let $R$ be a Noetherian ring that is $F$-finite or essentially of finite type over an excellent semilocal ring. Let $R\rightarrow S$ be a homomorphism of Noetherian rings and suppose that $N\subseteq M$ are arbitrary $R$-modules. Let $u\in N_M^*$, then $1\otimes u\in \langle S\otimes_R N\rangle_{S\otimes_R M}^*$.
\end{lem}

Note that by Kunz's Theorem, if $R$ is an $F$-finite ring of characteristic $p$, then $R$ is excellent. And a local ring essentially of finite type over an excellent semilocal ring is still excellent local, so we just need to assume $R$ to be excellent local. This is how our mild condition comes out.

\smallskip

Now we are ready to prove the tight closure version of Lemma \ref{lemDF}.

\begin{lem}\label{lemDT}
Let $(R,\mathfrak{m})$ be a Noetherian excellent local ring of characteristic $p>0$, and $I$ an $R$-ideal. Let $S=R/I$ and $J$ an $R$-ideal that contains $I$. If for $\bar{J}$ the image of $J$ in $S$ we have $\bar{J}$ is tightly closed as an $S$-ideal, then $J$ is tightly closed as an $R$-ideal, i.e. every ideal of $R$ that contains $I$ is tightly closed.
\end{lem}
\begin{proof}
Let $\varphi:R\rightarrow R/I$ be the natural projection map. Let $x\in J^*$, then by letting $R=G$, $R/I=M$, $\bar{J}=N$ and $J=H$ as in Lemma \ref{lemH1}, we can get that $\varphi(x)=\bar{x}\in \bar{J}_R^*$, where $\bar{J}_R^*$ is viewed as the tight closure of $\bar{J}$ as an $R$-ideal. Next by Lemma \ref{lemH2}, $\bar{x}=\bar{x}\otimes 1\in \bar{J}_S^*$, where this $\bar{J}_S^*$ is viewed as the tight closure of $\bar{J}$ as an $S$-ideal. By our condition, $\bar{J}_S^*=\bar{J}$, then $\bar{x}\in\bar{J}$, hence $x\in J$.
\end{proof}

Now we are ready to prove Proposition \ref{prop2}.

\begin{proof}[Proof of Proposition \ref{prop2}]
The proof is nearly the same as the proof of Proposition \ref{prop1}, except that we need to substitute Frobenius closed (or Frobenius closure) by tightly closed (or tight closure). Note that Theorem \ref{Alem1.2} substitutes Theorem \ref{LG(1)(2)} and Lemma \ref{lemDT} substitutes Lemma \ref{lemDF} in this proof.
\end{proof}

\bigskip

\section{Generalized Deformation Problems for $F$-injectivity and $F$-rationality}
In this section we use Proposition \ref{prop1} to solve the generalized deformation problem for $F$-injectivity (see Theorem \ref{ThmF-in}) and use Proposition \ref{prop2} to solve the generalized deformation problem for $F$-rationality (see Theorem \ref{ThmF-r}). Firstly we focus on $F$-injectivity. To do this, we need the following.
\smallskip

\begin{prop}\label{propQS}
Assume that $(R,\mathfrak{m})$ is an $F$-injective local ring of characteristic $p>0$. Then every ideal generated by a regular sequence is Frobenius closed.
\end{prop}

\begin{rem}
This proposition is almost the same as \cite[Proposition 3.13]{QS}, except that our proposition remove the condition that $R$ is $F$-finite. Actually this condition can be removed since in the original proof of \cite[Proposition 3.13]{QS}, the $F$-finiteness assumption is used in \cite[Lemma 3.11]{QS}, which says that if $(R,\mathfrak{m},K)$ is an $F$-injective local ring and $R$ is $F$-finite, then $R_P$ is $F$-injective for all $P\in \spec(R)$. But this statement of \cite[Lemma 3.11]{QS} does not require $R$ to be $F$-finite according to \cite[Theorem 4.13]{MP}.
\end{rem}

Now we are ready to prove our first main Theorem \ref{ThmF-in}.

\begin{thm}\label{ThmF-in}
Let $(R,\mathfrak{m})$ be a Noetherian local Cohen-Macaulay ring of characteristic $p>0$, $\dim\,R=d$. Suppose that there exists $I$ an $R$-ideal such that $\pd_R R/I<\infty$ and $R/I$ is $F$-injective, then $R$ is $F$-injective.
\end{thm}

\begin{proof}
Since $R/I$ is $F$-injective, then by Proposition \ref{propQS}, every ideal generated by a regular sequence of $R/I$ is Frobenius closed. Then by Proposition \ref{prop1}, every regular sequence on $R$ is Frobenius closed. Since $R$ is Cohen-Macaulay, that means every parameter ideal is Frobenius closed. Finally by \cite[Corollary 3.9]{QS}, we get that $R$ is $F$-injective.
\end{proof}

For the remaining of this section, we solve the generalized deformation problem for $F$-rationality under mild assumptions.

\begin{thm}\label{ThmF-r}
Let $(R,\mathfrak{m})$ be a Noetherian excellent local ring of characteristic $p>0$, and $\dim\,R=d$. Suppose there exists $I$ an $R$-ideal such that $\pd_R R/I<\infty$ and $R/I$ is $F$-rational, then $R$ is $F$-rational.
\end{thm}

\begin{proof}
 By Definition \ref{def2.7} and Remark \ref{rem2.8}, $F$-rationality means the ring is Cohen-Macaulay and every parameter ideal is tightly closed. Since $R/I$ is $F$-rational, it is Cohen-Macaulay, so every maximal regular sequence on $R/I$ is tightly closed. So by Proposition \ref{prop2}, every regular sequence on $R$ is tightly closed. Once we show that $R$ is Cohen-Macaulay, we can get that every parameter ideal is tightly closed. So it suffices to show that $R$ is Cohen-Macaulay. Since $R/I$ is $F$-rational, it is Cohen-Macaulay by Remark \ref{rem2.8}. Hence by \cite[Theorem 2.1]{AFH}, we have that $R$ is Cohen-Macaulay.
\end{proof}

\bigskip
\section{Generalized Deformation Problems for Serre's Conditions}
To prove that a ring $R$ is $R_k$ or $S_k$, it is natural to suppose that we have a prime ideal $P$ of $R$ that satisfies some conditions (For $R_k$ it is $\height\,P\leqslant k$, for $S_k$ it is $\depth\,R_P\leqslant k$). And usually the case $I\subseteq P$ is kind of easy to deal with. In this section we set up two lemmas for the case $I\nsubseteq P$ which says that if $P+I$ is $\mathfrak{m}$-primary in a Noetherian local ring $(R,\mathfrak{m})$, then $\pd_R R/I\geqslant \dim\,R/P$ and $\pd_R R/I\geqslant \depth\,R-\depth\,R_P$. The former helps us to prove the result for $R_k$ and the latter helps us to prove the result for $S_k$. And both of them help us to prove the result for $R_k+S_{k+1}$. Then we get the conclusions for normal rings and reduced rings as corollaries. At last we cite Auslander's Zerodivisor Conjecture which is a theorem by Robert (see \cite{R2}) and the corollary of generalized deformation problem for reduced rings to prove the theorem for domains.

\begin{lem}\label{defreg}\cite[Theorem 6.1(1)]{AFH}
Let $(R,\mathfrak{m})$ be a Noetherian local ring, $I$ an $R$-ideal, and $\pd_R R/I<\infty$. If $R/I$ is regular, then $R$ is regular.
\end{lem}

Next, we list the New Intersection Theorem to get a sequence of inequalities (Corollary \ref{inequality}) for an $R$-ideal $I$ of finite projective dimension, which is used in the proof of our first important theorem for $R_k$ (Theorem \ref{r}) in this section.

\begin{thm}[New Intersection Theorem]\label{NIT}\cite[Theorem 13.4.1]{R}
Let $(R,\mathfrak{m})$ be a local ring and $\dim\,R=d$. Let
$$F_{\bullet}:0\rightarrow F_k\rightarrow\ldots\rightarrow F_0\rightarrow 0$$
be a nonexact complex with $F_i$ free and $l(H_i(F_{\bullet}))<\infty$ for any $i$. Then $k\geqslant d$.
\end{thm}

\begin{prop}[Hochster \cite{H1}, Peskine-Szpiro \cite{PS} and Roberts \cite{R1}]\label{propdim}
Let $(R,\mathfrak{m})$ be a Noetherian local ring, $I$ an $R$-ideal such that $\pd_R R/I<\infty$, then $\dim\, R\leqslant \dim\,R/I+\pd_R R/I$.
\end{prop}
\begin{proof}
Let $s=\dim\, R/I$. Let $x_1,\ldots,x_s$ be a system of parameters of $R/I$, and let
$$F_{\bullet}:0\rightarrow F_g\rightarrow\ldots\rightarrow F_0\rightarrow 0$$
be a minimal free resolution of $R/I$, where $g=\pd_R R/I$. Consider $C_{\bullet}=F_{\bullet}\otimes_R K_{\bullet}(\underline{x})$. Since $F_{\bullet}$ is quasi-isomorphic with the complex
$$F'_{\bullet}:\ldots\rightarrow 0\rightarrow 0\rightarrow R/I\rightarrow 0$$
where $R/I$ is on the position of homological degree $0$ of the complex, we have that
$$H_i(C_{\bullet})=H_i(K_{\bullet}(\underline{x})\otimes_R F_{\bullet})\cong H_i(K_{\bullet}(\underline{x})\otimes_R F'_{\bullet})\cong H_i(K_{\bullet}(\underline{x})\otimes_R R/I)\cong H_i(\underline{x};R/I)$$
which is annihilated by $(\underline{x})+I$. Since $(\underline{x})+I$ is an $\mathfrak{m}$-primary ideal, we have that $l(H_i(C_{\bullet}))<\infty$. And since $F_{\bullet}$ has length $g$ and $K_{\bullet}(\underline{x})$ has length $s$, we have that $C_{\bullet}$ has length $g+s$. So $C_i\neq 0$ for $0\leqslant i\leqslant \pd_R R/I+s$. Also $H_0(C_{\bullet})\cong H_0(\underline{x}; R/I)\cong R/(I+(\underline{x}))\neq 0$. Now by Theorem \ref{NIT}, we have that $g+s\geqslant \dim\,R$, i.e. $\pd_R R/I+\dim\, R/I\geqslant \dim\, R$.
\end{proof}

Hence we have the following corollary:

\begin{cor}\label{inequality}
Let $(R,\mathfrak{m})$ be a Noetherian local ring, $I$ an $R$-ideal with $\pd_R R/I<\infty$. Then $\depth_I R\leqslant \height\,I\leqslant \dim\,R-\dim\,R/I\leqslant\pd_R R/I$. In particular, if $I$ is perfect, all these become equal.
\end{cor}

\begin{lem}\label{perfect}
Let $(R,\mathfrak{m})$ be a Noetherian local ring, $I$ a perfect $R$-ideal with $\pd_R R/I<\infty$. Then for any $P\in\spec(R)$, we have that $IR_P$ is a perfect $R_P$-ideal.
\end{lem}
\begin{proof}
We want to show $\grade(IR_P,R_P)=\pd_{R_P}R_P/IR_P$, i.e. we just need to show $\grade(IR_P,R_P)\geqslant\pd_{R_P}R_P/IR_P$ by Corollary \ref{inequality}.
\smallskip
\newline\indent By \cite[Corollary on Page 114]{H}, $\grade(IR_P,R_P)\geqslant\grade\,I$. So we have
\begin{equation*}
\grade(IR_P,R_P)\geqslant\grade\,I=\pd_R R/I\geqslant\pd_{R_P}R_P/IR_P.
\end{equation*}
\end{proof}

\begin{lem}\label{new1}
Let $(R,\mathfrak{m})$ be a Noetherian local ring, $I$ an $R$-ideal with $\pd_R R/I<\infty$, $P$ is a prime ideal of $R$ with $I\nsubseteq P$, and $P+I$ is $\mathfrak{m}-primary$, then $\dim\,R/P\leqslant\pd_R R/I$. 
\end{lem}
\begin{proof}
Since $\pd_R R/I<\infty$, we can write out a minimal free resolution
$$0\rightarrow F_n\rightarrow\ldots\rightarrow F_0\rightarrow 0$$
of $R/I$ in terms of free $R$-modules $F_i$. Tensor it with $R/P$, we get that
\begin{equation}
0\rightarrow F_n\otimes_R R/P\rightarrow\ldots\rightarrow F_0\otimes_R R/P\rightarrow 0.\label{1}
\end{equation}
\indent Since $\tor_0^R(R/I,R/P)=R/I\otimes_R R/P=R/(I+P)\neq 0$, the complex \eqref{1} is not exact. And each $F_i\otimes_R R/P$ is a free $R/P$-free module. Since each $\tor_i^R(R/I,R/P)$ can be killed by $P$ and $I$, it can be killed by $P+I$. And since $P+I$ is $\mathfrak{m}$-primary, $\tor_i^R(R/I,R/P)$ can be killed by a power of $\mathfrak{m}$. So if we regard $\tor_i^R(R/I,R/P)$ as an $R/P$-module, it can be killed by a power of $\mathfrak{m}/P$, i.e. $\mathnormal{l}_{R/P}(\tor_i^R(R/I,R/P))<\infty$. Hence by Theorem \ref{NIT}, we get that $n\geqslant \dim\,R/P$, i.e. $\pd_R R/I\geqslant \dim\,R/P$.
\end{proof}

Now we can begin the proof of $R_k$ for the generalized deformation problem.

\begin{thm}\label{r}
Let $(R,\mathfrak{m})$ be a Noetherian local ring which is equidimensional and catenary, $I$ a perfect $R$-ideal with $\pd_R R/I<\infty$. If $R/I$ is $R_k$, then $R$ is $R_k$.
\end{thm}

\begin{proof}
Let $P\in\spec(R)$ with $\height\,P\leqslant k$, we want to show $R_P$ is regular.
\smallskip
\newline\indent If $I\subseteq P$, then $\height_{R/I}P(R/I)\leqslant \height\,P\leqslant k$. Hence $(R/I)_{P(R/I)}\cong R_P/IR_P$ is regular. And for that $\pd_{R_P}R_P/IR_P\leqslant \pd_R R/I<\infty$, by Lemma \ref{defreg} we have that $R_P$ is regular.
\smallskip
    \newline\indent If $I\nsubseteq P$, then choose a minimal prime $Q$ over $P+I$, we claim that $\height_{R/I}Q/I\leqslant k$. To show this claim, it suffices to show that $\dim\,R_Q/IR_Q\leqslant \height_{R_Q}\,PR_Q$ because $\dim R_Q/IR_Q=\height_{R/I} Q/I$ and $\height_{R_Q}PR_Q=\height\,P$ by the condition that $R$ is equidimensional and catenary. Since localization preserves equidimensional and catenary properties and it also preserves perfect ideal by Lemma \ref{perfect} and finite projective resolutions, we may replace $R_Q$ by $(R,\mathfrak{m})$, then this goal becomes $\dim\, R/I\leqslant \height\,P$. And since $\dim\,R/I=\dim\,R-\height\,I$ by Corollary \ref{inequality}, it suffices to show $\dim\,R-\height\,I\leqslant\height\,P$. By Corollary \ref{inequality}, we have $\height\,I=\pd_R R/I$. By Proposition \ref{e+c}, we have $\dim\,R/P=\dim\,R-\height\,P$. So it suffices to show $\dim\,R/P\leqslant\pd_R R/I$, which is true by Lemma \ref{new1}.
    \smallskip
    \newline\indent Hence the claim is true. Since $R/I$ is $R_k$, we have $(R/I)_{Q/I}\cong R_Q/IR_Q$ is regular. And $\pd_{R_Q} R_Q/IR_Q\leqslant \pd_R R/I<\infty$, then by Lemma \ref{defreg}, $R_Q$ is regular. And since $R_P$ is a localization of $R_Q$, we get that $R_P$ is regular.
\end{proof}

\begin{lem}\label{new2}
Let $(R,\mathfrak{m})$ be a Noetherian local ring, $I$ an $R$-ideal with $\pd_R R/I<\infty$. Let $P$ be a prime ideal of $R$ such that $I\nsubseteq P$ and $P+I$ is $\mathfrak{m}$-primary, then $\depth\,R\leqslant\depth\,R_P+\pd_R R/I$.
\end{lem}
\begin{proof}
By Lemma \ref{new1}, we already have $\dim\,R/P\leqslant\pd_R R/I$. And we also have $\depth_P R\leqslant\depth\, R_P$, so it suffices to show
\begin{equation}
\depth\,R-\depth_P R\leqslant\dim\,R/P.\label{5}
\end{equation}
\indent Let $x_1,\ldots,x_i$ be a maximal regular sequence contained in $P$, and extend it to $x_1,\ldots,x_i,x_{i+1},\ldots,x_n$ which is maximal regular contained in $m$. Then $n-i=\depth\,R-\depth_P R$, and $x_{i+1},\ldots,x_n$ is maximal regular on $R/(x_1,\ldots,x_i)$, so $n-i=\depth\,R/(x_1,\ldots,x_i)$.
\smallskip
\newline\indent Since $x_1,\ldots,x_i$ is maximal regular contained in $P$, $P$ does not contain any regular element on $R/(x_1,\ldots,x_i)$, hence $P$ is composed of zerodivisors of $R/(x_1,\ldots,x_i)$, hence $P$ is contained in the union of associated primes of $R/(x_1,\ldots,x_i)$, i.e. $P\subseteq Q$ for some $Q\in \ass(R/(x_1,\ldots,x_i))$. Hence $n-i=\depth\,R/(x_1,\ldots,x_i)\leqslant\dim\,R/Q\leqslant\dim\,R/P$.
\end{proof}

\begin{thm}\label{s}
Let $(R,\mathfrak{m})$ be a Noetherian local ring, $I$ an $R$-ideal with $\pd_R R/I<\infty$. If $R/I$ is $S_k$, then $R$ is $S_k$.
\end{thm}

\begin{proof}
It suffices to consider those $P\in\spec(R)$ such that $\depth\,R_P<k$, and show that $\depth\,R_P=\height\,P$.
\smallskip
\newline\indent If $I\subseteq P$, then $\depth\,(R/I)_P=\depth\,R_P-\pd_{R_P}R_P/IR_P\leqslant\depth\,R_P<k$. Since $R/I$ is $S_k$, then
\begin{equation}
\height_{R/I}P(R/I)=\depth\,(R/I)_{P(R/I)}.\label{2}
\end{equation}
\indent We want to show $\height\,P=\depth\,R_P$. Since $\pd_R R/I<\infty$ implies that $\pd_{R_P}R_P/IR_P<\infty$, we may replace $R_P$ by $R$, then \eqref{2} becomes
\begin{equation*}
\dim\,R/I=\depth\,R/I,
\end{equation*}
and we want to show $\dim\,R=\depth\,R$. This is true by \cite[Theorem 2.1]{AFH}.
\newline\indent If $I\nsubseteq P$, let $Q$ be a minimal prime over $P+I$. We claim that $\depth\,(R/I)_Q<k$. Since we have that $\pd_{R_Q} R_Q/IR_Q<\infty$, then by Lemma \ref{new2} we have that $\depth\,R_Q\leqslant\depth\,R_P+\pd_{R_Q}R_Q/IR_Q$. And then by Auslander-Buchsbaum's formula we get that $\depth\,(R/I)_Q=\depth\,R_Q-\pd_{R_Q}R_Q/IR_Q\leqslant\depth\,R_P<k$, hence the claim is true.
\smallskip
\newline\indent Since $R/I$ is $S_k$, we have $\depth\,(R/I)_Q=\height_{R/I}Q(R/I)$, i.e. $R_Q/IR_Q$ is Cohen-Macaulay. So by \cite[Theorem 2.1]{AFH} again, we have that $R_Q$ is Cohen-Macaulay. And since $R_P$ is a localization of $R_Q$, we have that $R_P$ is Cohen-Macaulay, i.e. $\depth\,R_P=\height\,P$.
\end{proof}

\begin{thm}\label{sandr}
Let $(R,\mathfrak{m})$ be a Noetherian local ring, $I$ an $R$-ideal with $\pd_R R/I<\infty$. If $R/I$ is $S_{k+1}$ and $R_k$, then $R$ is also $S_{k+1}$ and $R_k$.
\end{thm}
\begin{proof}
By Theorem \ref{s}, we have already proved that $R/I$ being $S_{k+1}$ would imply $R$ being $S_{k+1}$. So it remains to prove that $R$ is $R_k$. Let $P\in\spec (R)$ with $\height\,P\leqslant k$, we need to show that $R_P$ is regular.
\smallskip
\newline\indent If $I\subseteq P$, then $\height_{R/I} P/I\leqslant\height\, P\leqslant k$, hence $(R/I)_{P/I}\cong R_P/IR_P$ is regular. Since $\pd_{R_P} R_P/IR_P<\infty$, by Lemma \ref{defreg} we have that $R_P$ is regular. This part is the same as the second paragraph of the proof of Theorem \ref{r}.
\smallskip
\newline\indent If $I\nsubseteq P$, we choose a minimal prime $Q$ over $P+I$, then we claim that $\depth\,(R/I)_{Q/I}\leqslant k$. By Auslander-Buchsbaum's formula, $\depth\,(R/I)_{Q/I}=\depth\,R_Q-\pd_{R_Q}R_Q/IR_Q$, so it suffices to show
\begin{equation*}
\depth\,R_Q-\pd_{R_Q}R_Q/IR_Q\leqslant\height_{R_Q}PR_Q=\height\,P\leqslant k.
\end{equation*}
\indent Note that we have the following
\begin{equation*}
\begin{split}
\depth\,R_Q&\leqslant\depth_{PR_Q}R_Q+\dim\,R_Q/PR_Q\\
&\leqslant\height_{R_Q}PR_Q+\dim\,R_Q/PR_Q\\
&\leqslant\height_{R_Q}PR_Q+\pd_{R_Q}R_Q/IR_Q,
\end{split}
\end{equation*}
where the first row holds because of \eqref{5} in Lemma \ref{new2}, and the third row holds because of Lemma \ref{new1}. Hence the claim is true, i.e. $\depth\,(R/I)_{Q/I}\leqslant k$.
\smallskip
\newline\indent Since $R/I$ is $S_{k+1}$, we have that
\begin{equation*}
k\geqslant\depth\,(R/I)_{Q/I}\geqslant\min \{k+1, \height_{R/I}Q/I\},
\end{equation*}
which gives that $k\geqslant\height_{R/I}Q/I$. Then since $R/I$ is $R_k$, we have that $(R/I)_{Q/I}$ is regular, i.e. $R_Q/IR_Q$ is regular. Then by Lemma \ref{defreg}, $R_Q$ is regular. Hence $R_P$ is regular since $R_P$ is a localization of $R_Q$.
\end{proof}

Since it is well known that a ring is reduced if and only if it is $R_0$ and $S_1$, and that a ring is normal if and only if it is $R_1$ and $S_2$, we naturally get the following two corollaries.

\begin{cor}\label{reduced}
Let $(R,\mathfrak{m})$ be a Noetherian local ring, $I$ an $R$-ideal with $\pd_R R/I<\infty$. If $R/I$ is reduced, then $R$ is reduced.
\end{cor}

\begin{cor}\label{normal}
Let $(R,\mathfrak{m})$ be a Noetherian local ring, $I$ an $R$-ideal with $\pd_R R/I<\infty$. If $R/I$ is normal, then $R$ is normal.
\end{cor}

At last we come to our final case for domains. We need a result in the following as a lemma:

\begin{lem}\cite[Conjecture 1]{R2}\label{zdconj}
Let $(R,\mathfrak{m})$ be a Noetherian local ring, $M$ is a nonzero $R$-module of finite projective dimension. If $x\in R$ is regular on $M$, then $x$ is regular on $R$.
\end{lem}

This conjecture is called Auslander's Zerodivisor Conjecture, but it has already become a theorem due to Peskine-Szpiro's Intersection Theorem (See \cite{R2}), so we can use it directly. Now we are ready to prove our last theorem about the generalized deformation question for domains.

\begin{thm}\label{domain}
Let $(R,\mathfrak{m})$ be a Noetherian local ring, $P$ an $R$-ideal such that $\pd_R R/P<\infty$. If $R/P$ is a domain, then $R$ is a domain.
\end{thm}
\begin{proof}
Since $R/P$ is a domain, it is reduced. Then by Corollary \ref{reduced}, we have that $R$ is reduced. So to prove that $R$ is a domain, it suffices to show that $R$ has a unique minimal prime ideal.
\smallskip
\newline\indent Let $M=R/P$ as in Lemma \ref{zdconj}, then this lemma implies that if $x\notin P$, then $x$ is regular on $R$. Equivalently, if $x$ is a zerodivisor on $R$, then $x\in P$. So all associated primes of $R$ are contained in the prime ideal $P$. Since $R$ is reduced, this is equivalent to saying that all minimal primes of $R$ are contained in $P$.
\smallskip
\newline\indent Since $\pd_R R/P<\infty$, we have that $\pd_{R_P}R_P/PR_P<\infty$, i.e. $R_P$ is a regular local ring. Hence $R_P$ is a domain, which says that $R_P$ has a unique minimal prime ideal. Since all minimal primes of $R$ are contained in $P$, due to the one-to-one correspondence between the minimal primes of $R$ and the minimal primes of $R_P$, we know that $R$ has a unique minimal prime. Hence $R$ is a domain.
\end{proof}

\section*{Acknowledgement}
Here I would like to thank Professor Linquan Ma for introducing the topics to me and giving many really important comments on the work. The author is partially supported by NSF-DMS grant 1901672. The author thanks the referee for valuable comments.
\bigskip

\end{document}